\documentclass[10pt,a4paper,twoside]{article}
\usepackage{amsmath,amsthm,amssymb}
\usepackage[usenames,dvipsnames]{xcolor}
\usepackage{hyperref}
\definecolor{DarkPurple}{rgb}{0.40,0.0,0.20}
\hypersetup{colorlinks=true,citecolor=NavyBlue,linkcolor=DarkPurple,urlcolor=NavyBlue, breaklinks=true}
\usepackage{color} 

\usepackage{times}
\usepackage{enumerate}
\usepackage{a4, amsmath}
\usepackage{mathtools}
\usepackage{amssymb}
\usepackage{epsfig}
\usepackage{graphics}
\usepackage{amsthm, amscd, mathdots}
\usepackage{enumerate}
\usepackage{url}
\usepackage{cancel}
\usepackage{cleveref}
\usepackage{csquotes}
\usepackage{color}
\usepackage{datetime}
\usepackage{todonotes}
\usepackage{stmaryrd}

\theoremstyle{definition}
\newtheorem{defi}{Definition}[section]
\theoremstyle{plain}
\newtheorem{prop}[defi]{Proposition}
\newtheorem{lem}[defi]{Lemma}
\newtheorem{stel}[defi]{Theorem}
\newtheorem{gev}[defi]{Corollary}

\newtheorem{ques}[defi]{Question}
\newtheorem*{stel*}{Theorem}
\newtheorem*{gev*}{Corollary}
\theoremstyle{remark}
\newtheorem{opm}[defi]{Remark}

\newcommand{\La}{\mathcal{L}}
\newcommand{\nat}{\mathbb{N}}
\newcommand{\zz}{\mathbb{Z}}
\newcommand{\rr}{\mathbb{R}}

\newcommand{\qq}{\mathbb{Q}}

\DeclareMathOperator{\charac}{char}

\DeclareMathOperator{\Frac}{Frac}

\DeclareMathOperator{\Trd}{Trd}

\DeclareMathOperator{\Nrd}{Nrd}

\DeclareMathOperator{\Odd}{\mathsf{Odd}}
\DeclareMathOperator{\Neg}{\mathsf{Neg}}
\def\YEAR{2021}\newcount\VOL\VOL=\YEAR\advance\VOL by-1995
\def\firstpage{1851}\def\lastpage{1869}
\def\received{November 13, 2020}\def\revised{November 20, 2021}
\def\communicated{Michael Rathjen}

\makeatletter
\def\magnification{\afterassignment\m@g\count@}
\def\m@g{\mag=\count@\hsize6.5truein\vsize8.9truein\dimen\footins8truein}
\makeatother

\font\eightrm=cmr8
\font\caps=cmcsc10                    % Theorem, Lemma etc
\font\Caps=cmcsc10 scaled \magstep1   % Title
\font\scaps=cmcsc8

\makeatletter
\@specialpagefalse\headheight=8.5pt
\def\DocMath{{\def\th{\thinspace}\scaps Documenta Math.}}
\renewcommand{\@oddfoot}{\hfill\scaps Documenta Mathematica 
    \number\VOL\  (\number\YEAR) \number\firstpage--\lastpage\hfill}
\renewcommand{\@evenfoot}{\ifnum\thepage>\lastpage\hfill\scaps
    Documenta Mathematica \number\VOL\  (\number\YEAR)\hfill\else\@oddfoot\fi}%

\renewcommand{\@evenhead}{%
    \ifnum\thepage>\lastpage\rlap{\thepage}\hfill%
    \else\rlap{\thepage}\slshape\leftmark\hfill{\caps\SAuthor}\hfill\fi}%
\renewcommand{\@oddhead}{%
    \ifnum\thepage=\firstpage{\DocMath\hfill\llap{\thepage}}%
    \else{\slshape\rightmark}\hfill{\caps\STitle}\hfill\llap{\thepage}\fi}%
\makeatother

\def\TSkip{\bigskip}
\newbox\TheTitle{\obeylines\gdef\GetTitle #1
\ShortTitle  #2
\SubTitle    #3
\Author      #4
\ShortAuthor #5
\EndTitle
{\setbox\TheTitle=\vbox{\baselineskip=20pt\let\par=\cr\obeylines%
\halign{\centerline{\Caps##}\cr\noalign{\medskip}\cr#1\cr}}%
	\copy\TheTitle\TSkip\TSkip%
\def\next{#2}\ifx\next\empty\gdef\STitle{#1}\else\gdef\STitle{#2}\fi%
\def\next{#3}\ifx\next\empty%
    \else\setbox\TheTitle=\vbox{\baselineskip=20pt\let\par=\cr\obeylines%
    \halign{\centerline{\caps##} #3\cr}}\copy\TheTitle\TSkip\TSkip\fi%
\centerline{\caps #4}\TSkip\TSkip%
\def\next{#5}\ifx\next\empty\gdef\SAuthor{#4}\else\gdef\SAuthor{#5}\fi%
\ifx\received\empty\relax
    \else\centerline{\eightrm Received: \received}\fi%
\ifx\revised\empty\TSkip%
    \else\centerline{\eightrm Revised: \revised}\TSkip\fi%
\ifx\communicated\empty\relax
    \else\centerline{\eightrm Communicated by \communicated}\fi\TSkip\TSkip%
\catcode'015=5}}\def\Title{\obeylines\GetTitle}
\def\Abstract{\begingroup\narrower
    \parskip=\medskipamount\parindent=0pt{\caps Abstract. }}
\def\EndAbstract{\par\endgroup\TSkip}

\long\def\MSC#1\EndMSC{\def\arg{#1}\ifx\arg\empty\relax\else
     {\par\narrower\noindent%
     2020 Mathematics Subject Classification: #1\par}\fi}

\long\def\KEY#1\EndKEY{\def\arg{#1}\ifx\arg\empty\relax\else
	{\par\narrower\noindent Keywords and Phrases: #1\par}\fi\TSkip}

\newbox\TheAdd\def\Addresses{\vfill\copy\TheAdd\vfill
    \ifodd\number\lastpage\vfill\eject\phantom{.}\vfill\eject\fi}
{\obeylines\gdef\GetAddress #1
\Address #2 
\Address #3
\Address #4
\EndAddress
{\def\xs{5.1truecm}\parindent=0pt
\setbox0=\vtop{{\obeylines\hsize=\xs#1\par}}\def\next{#2}
\ifx\next\empty % 1 address
     \setbox\TheAdd=\hbox to\hsize{\hfill\copy0\hfill}
\else\setbox1=\vtop{{\obeylines\hsize=\xs#2\par}}\def\next{#3}
\ifx\next\empty % 2 addresses
     \setbox\TheAdd=\hbox to\hsize{\hfill\copy0\hfill\copy1\hfill}
\else\setbox2=\vtop{{\obeylines\hsize=\xs#3\par}}\def\next{#4}
\ifx\next\empty\ % 3 addresses
     \setbox\TheAdd=\vtop{\hbox to\hsize{\hfill\copy0\hfill\copy1\hfill}
                \vskip20pt\hbox to\hsize{\hfill\copy2\hfill}}
\else\setbox3=\vtop{{\obeylines\hsize=\xs#4\par}}
     \setbox\TheAdd=\vtop{\hbox to\hsize{\hfill\copy0\hfill\copy1\hfill}
	        \vskip20pt\hbox to\hsize{\hfill\copy2\hfill\copy3\hfill}}
\fi\fi\fi\catcode'015=5}}\gdef\Address{\obeylines\GetAddress}

\def\LOCAL{\jobname.files}
\begin{document}
%%%%% ------------- fill in your data below this line  -------------------
%%%%%    The following lines \Title ... \EndAddress must ALL be present
%%%%%    and in the given order.
\Title
Universally Defining Finitely Generated Subrings 
of Global Fields
%%%%%    Put here the title. Line breaks will be recognized. 
\ShortTitle 
Universally Defining Subrings of Global Fields
%%%%%    Running title for odd numbered pages, ONE line, please. 
%%%%%    If none is given, \Title will be used instead.          
\SubTitle   
%%%%%    A possible subtitle goes here.
\Author 
Nicolas Daans
%%%%%    Put here name(s) of authors. Line breaks will be recognized.  
\ShortAuthor 
N. Daans
%%%%%%   Running title for even numbered pages, ONE line, please. 
%%%%%%   If none is given, \Author will be used instead.          
\EndTitle
\Abstract 
It is shown that any finitely generated subring of a global field has a universal first-order definition in its fraction field. This covers Koenigsmann's result for the ring of integers and its subsequent extensions to rings of integers in number fields and rings of $S$-integers in global function fields of odd characteristic. In this article a proof is presented which is uniform in all global fields, including the characteristic two case, where the result is entirely novel. Furthermore, the proposed method results in universal formulae requiring significantly fewer quantifiers than the formulae that can be derived through the previous approaches.
%%%%%    Put here the abstract of your manuscript.
\EndAbstract
\MSC 
Primary 11U99; Secondary 11R52
%%%%%    2010 Mathematics Subject Classification: 
\EndMSC
\KEY 
Diophantine set, definability, quaternion algebra, Hilbert 10
%%%%%    Keywords and Phrases:     
\EndKEY
%%%%%    All 4 \Address lines below must be present. To center the last
%%%%%    entry, no empty lines must be between the following \Address
%%%%%    and \EndAddress lines.
\Address
Nicolas Daans
Universiteit Antwerpen
Departement Wiskunde
Campus Middelheim - G, M.G.105
Middelheimlaan 1
2020 Antwerpen
Belgium
nicolas.daans@uantwerpen.be
\Address
\Address
\Address
\EndAddress
%%
%%       Make sure the last tex command in your manuscript
%%       before the first \end{document} is the command  \Addresses
%%
%%---------------------Here the prologue ends---------------------------------
%%--------------------Here the manuscript starts------------------------------

\section{Introduction}
It was recently shown by Koenigsmann that there is a universal definition of $\zz$ in $\qq$ in the first-order language of rings \cite{Koe16}. That is, he showed that there exist a natural number $m$ and a polynomial $F \in \qq[X, Y_1, \ldots, Y_m]$ such that
\begin{displaymath}
\mathbb{Z} = \lbrace x \in \mathbb{Q} \mid \forall y_1, \ldots, y_m \in \mathbb{Q} : F(x, y_1, \ldots, y_m) \neq 0 \rbrace.
\end{displaymath}
This builds on earlier work by Poonen, who had derived an $\forall\exists$-formula defining $\zz$ in $\qq$ \cite{Poo09}. That $\zz$ has a first-order definition in $\qq$ at all was already known long before and first shown by Robinson \cite{Rob49}.

The purpose of this paper is to present a variation of Koenigsmann's construction. This adaptation not only shortens the proof and yields a simpler formula by removing the need for many case distinctions, it also generalises directly to other global fields. Before stating our results in full generality, let us illustrate how the method applies for universally defining $\zz$ in $\qq$.

Let $\mathbb{P}$ be the set of prime numbers. For $p \in \mathbb{P}$, we denote by $v_p$ the $p$-adic valuation on~$\qq$, by $\zz_{(p)}$ the corresponding valuation ring of $\qq$, and by $p\zz_{(p)}$ its maximal ideal.

One can easily obtain a universal definition of $\zz$ in $\qq$ once one has found an existential definition of $\bigcup_{p \in \mathbb{P}} p\zz_{(p)}$ in $\qq$. Indeed, for $x \in \qq^\times$ one has
\begin{displaymath}
x \in \zz \enspace \Leftrightarrow \enspace x^{-1} \not\in \bigcup_{p \in \mathbb{P}} p\zz_{(p)}.
\end{displaymath}
To work our way towards an existential definition of $\bigcup_{p \in \mathbb{P}} p\zz_{(p)}$ in $\qq$, quaternion algebras over $\qq$ turn out to be a useful tool. Quaternion algebras were first introduced in this context in \cite{Poo09}. For a field $K$ with $\charac(K) \neq 2$ and $a, b \in K^\times$, write $(a, b)_K$ for the $K$-quaternion algebra with generators $i, j$ such that $i^2 = a$, $j^2 = b$ and $ij + ji = 0$. To a given $a, b \in \qq^\times$, we associate the set
\begin{displaymath}
\Delta_{a, b} = \lbrace p \in \mathbb{P} \mid (a, b)_{\qq_p} \text{ not split} \rbrace.
\end{displaymath}
This is a finite set of prime numbers also called the \emph{ramification set} of the quaternion algebra $(a, b)_\qq$; it clearly depends only on the isomorphism class of $(a, b)_\qq$ as a $\qq$-algebra. In \cref{sectlocalglobal} its computation and properties will be discussed. Now we define the subset
\begin{displaymath}
\Delta_{a, b}^c = \lbrace p \in \Delta_{a, b} \mid v_p(c) \text{ is odd} \rbrace
\end{displaymath}
for $a, b, c \in \qq^\times$. In \cref{sectsemilocal} it will be shown that the subset of $\qq$
\begin{displaymath}
J^c_{a, b} = \bigcap_{p \in \Delta^c_{a, b}} p\zz_{(p)}
\end{displaymath}
has an existential definition in $\qq$, uniformly in the parameters $a, b, c \in \qq^\times$. Here, we take the convention that $\bigcap \emptyset = \qq$. These sets were implicitly already introduced and given an existential definition by Koenigsmann \cite{Koe16}, building on earlier work by Poonen \cite{Poo09}. %We will give a simpler existential definition.

To obtain a universal definition of $\zz$ in $\qq$, it then remains to show that we can build an existential definition of $\bigcup_{p \in \mathbb{P}} p\zz_{(p)}$ by using the existentially definable sets $J_{a, b}^c$. Here we deviate substantially from the approach in \cite{Koe16}. 
Consider the following subset of $(\qq^{\times})^2$:
\begin{displaymath}
\Phi = \lbrace (1+4a^2, 2b) \mid a, b \in \zz_{(2)}^\times \rbrace.
\end{displaymath}
This set is existentially definable in $\qq$: to see this one can, for example, use that $2\zz_{(2)} = J^2_{2, 5}$. We will see that
\begin{equation}\label{eqintro}
\bigcup_{p \in \mathbb{P}} p\zz_{(p)} = 2\zz_{(2)} \cup \bigcup_{(x, y) \in \Phi} (J^{x}_{x, y} \cap J^{2y}_{x, y}).
\end{equation}
As the set on the right is existentially definable, this gives us the required existential definition for $\bigcup_{p \in \mathbb{P}} p\zz_{(p)}$ in $\qq$.

\medskip

Let us explain why the inclusion from right to left holds. Take an arbitrary pair $(x, y) \in \Phi$. It is not hard to see that $2 \in \Delta_{x, y}$ (see \Cref{pitchfork2}). Furthermore, as $x > 0$, it is a well-known corollary of the Quadratic Reciprocity Law that $\Delta_{x, y}$ contains an even number of elements (\Cref{Hilbertreciprocity}); in particular we must have $\Delta_{x, y} \neq \lbrace 2 \rbrace$, so there exists a prime number $p \in \Delta_{x, y} \setminus \lbrace 2 \rbrace$. This implies $(x, y)_{\qq_p}$ is non-split, so that either $v_p(x)$ or $v_p(y) = v_p(2y)$ is odd (see \Cref{necessaryconditionsplitlocal}); we conclude that $p \in \Delta_{x, y}^x \cup \Delta_{x, y}^{2y}$ and hence $(J^{x}_{x, y} \cap J^{2y}_{x, y}) \subseteq p\zz_{(p)}$. As this holds for arbitrary $(x, y) \in \Phi$, this shows the inclusion from right to left in (\ref{eqintro}).

\medskip
For the other inclusion, the key point is to show that, given an odd prime $p$, one can find a pair $(x, y) \in \Phi$ such that $\Delta_{x, y} = \lbrace 2, p \rbrace$. One possible approach is to first find a prime number $q \equiv 5 \bmod 8$ such that $p$ is a non-square modulo $q$. For such $q$ one has $\Delta_{q, 2p} = \lbrace 2, p \rbrace$ and $q = c^2 + 4d^2$ for certain $c, d \in \zz \setminus 2\zz$; hence one can set $x = q/c^2$, $y = 2p$. Having found $x$ and $y$ such that $\Delta_{x, y} = \lbrace 2, p\rbrace$, we see that $\Delta_{x, y}^x \cup \Delta_{x, y}^{2y} = \lbrace p \rbrace$ and as such $J^{x}_{x, y} \cap J^{2y}_{x, y} = p\zz_{(p)}$. This concludes the proof of the inclusion from left to right.

\medskip

With some adjustments a similar construction can be used in a more general context. For a global field $K$ (i.e. a number field or a function field in one variable over a finite field) and a finite set $S$ of valuations on $K$, define the \textit{ring of $S$-integers} as the intersection of all valuation rings of $K$ excluding those which are given by valuations in $S$. Our main result can be stated as follows:
\begin{stel*}
Let $K$ be a global field, $S$ a finite set of valuations on $K$. The ring of $S$-integers has a universal first-order definition in $K$ with $37$ quantifiers in the language of rings with constants for $K$.
\end{stel*}
Without the quantitative bound, large parts of this result were already established by Park \cite{Par13} for number fields and Eisentr{\"a}ger and Morrison \cite{Eis18} for global fields of odd characteristic. 
We recover these results with a uniform proof for all cases. This is because we reduce the number theoretic ingredients to just two statements. Firstly, that a quaternion algebra over a global field $K$ which is split over all embeddings into~$\rr$ ramifies at an even number of primes (a fact known in the case $K = \qq$ as Hilbert's Reciprocity Law and closely related to the Quadratic Reciprocity Law). Secondly that conversely, for every finite set of primes of even cardinality of a global field $K$, there exists a quaternion algebra over $K$ which ramifies precisely at these primes. To be able to also cover global fields of characteristic $2$ in this article, for which the question has so far remained undiscussed in the literature, we will switch to a characteristic-independent parametrisation of quaternion algebras due to Albert, which we will state and briefly elaborate on in \cref{sectquat}.

\medskip
The main theorem is first proven in \cref{sectS} without quantitative bounds. This way, steps which serve only to lower the number of quantifiers do not distract the reader from the structure of the argument. In \cref{sectquantcount} we zoom in on how to obtain the bound on the number of quantifiers. The original method of \cite{Koe16}, even just in the case 
\newpage 
of defining $\zz$ in $\qq$, yields definitions with several hundreds of quantifiers.\footnote{The published version of \cite{Koe16} does not explicitly count the number of quantifiers. In an earlier preprint available online (\url{https://archive.org/details/arxiv-1011.3424}) it is claimed that a universal formula with $418$ quantifiers is obtained. However, I was unable to replicate this count and reach the same final number.}
%In an earlier preprint of the current article, a universal definition with $50$ quantifiers was obtained, and a sketch given on how to reduce this to $38$.
The technique from the current article has been further improved in the case of defining~$\zz$ in $\qq$ in a recent preprint by Sun and Zhang \cite{ZhangSunZinQ}, leading to a definition with $32$ quantifiers.
%In the current article we incorporate their techniques (and results of the author with Fehm and Dittmann \cite{DDF}) to obtain the upper bound of $29$ in general.

Rings of $S$-integers in a global field $K$ are finitely generated subrings of $K$ which have~$K$ as their fraction field. In \cref{sectfingen} it will be explained how universal definability of arbitrary finitely generated subrings of $K$ with fraction field $K$ follows easily from universal definability of rings of $S$-integers. Here, however, we lose the bound on the number of quantifiers. Nevertheless, the most general result on universal definability of subrings of global fields we obtain can then be stated compactly as follows:
\begin{gev*}
Let $K$ be a global field. Any finitely generated subring $R$ of $K$ such that $K$ is the fraction field of $R$ has a universal first-order definition in $K$.
\end{gev*}

\subsection*{Acknowledgements}  I would like to thank Philip Dittmann for suggesting the approach in Section 8, as well as Jan Van Geel and Bjorn Poonen for pointing out \Cref{Dirichlet}. 

This work grew out of my master thesis. It will also be part of my PhD thesis, prepared at the University of Antwerp under the supervision of Karim Johannes Becher and Philip Dittmann.

This work was supported by Fonds Wetenschappelijk Onderzoek - Vlaanderen (FWO) through PhD Fellowship fundamental research 51581, and by the FWO Odysseus programme (project \emph{Explicit Methods in Quadratic Form Theory}).

\section{Logical preliminaries and notation}
Let $\nat$ denote the set of natural numbers, including $0$. We will further use the notation $2\nat = \lbrace 2n \mid n \in \nat \rbrace$.

We will use some basic standard terminology from mathematical logic when dealing with the syntax and semantics of statements in first-order language, as covered by many textbooks (e.g. \cite[Chapter II-III]{Ebb94}). We denote by $\La$ the first-order language of rings, that is the language consisting of three binary operation symbols $+, -, \cdot$ and two constant symbols $0, 1$. We denote by $\doteq$ the equality symbol in the language. Let $R$ be a commutative ring and let $\La_R$ denote the language of rings extended with constant symbols for the elements of $R$. We interpret a commutative $R$-algebra $K$ as an $\La_R$-structure via the action of $R$ on $K$. The identity map makes any commutative ring $R$ into an $\La_R$-structure in a canonical way.

Let $\varphi(X_1, \ldots, X_n)$ be an $\La_R$-formula in free variables $X_1, \ldots, X_n$. Given a commutative $R$-algebra $K$ and a tuple $(x_1, \ldots, x_n) \in K^n$, we write $K \models \varphi(x_1, \ldots, x_n)$ if and only if, after substituting $x_i$ for $X_i$ for each $i \in \lbrace 1, \ldots, n \rbrace$, the $\La_K$-statement $\varphi(x_1, \ldots, x_n)$ holds when evaluated in $K$. We call two $\La_R$-formulas $\varphi_1$ and $\varphi_2$ in free variables $X_1, \ldots, X_n$ \textit{equivalent} if for any $\La_R$-structure $K$ and for any $(x_1, \ldots, x_n) \in K^n$ we have $K \models \varphi_1(x_1, \ldots, x_n)$ if and only if $K \models \varphi_2(x_1, \ldots, x_n)$.

As we will only evaluate first-order formulas in commutative rings, there is no ambiguity in interpreting atomic $\La_R$-formulas as polynomial equalities with coefficients in~$R$. Furthermore, we may identify $\La$ with $\La_\zz$, as every ring is a $\zz$-algebra in a unique way. Up to equivalence, every existential $\La_R$-formula $\varphi(X_1, \ldots, X_n)$ in free variables $X_1, \ldots, X_n$ can be written as
\begin{equation}\label{posex}
\exists Y_1, \ldots, Y_m \bigvee_{i=1}^p\left(\left(\bigwedge_{j=1}^{q_i} f_{i, j} \doteq 0\right)\wedge\left(\bigwedge_{k=1}^{r_i}\neg(g_{i, k} \doteq 0) \right)\right)
\end{equation}
for some $m, p, q_i, r_i \in \nat$ and $f_{i, j}, g_{i, k} \in R[X_1, \ldots, X_n, Y_1, \ldots, Y_m]$ for all $i, j, k$. For the sake of brevity, we will call an existential $\La_R$-formula with $m$ quantifiers an $\exists_m\La_R$-\textit{formula}. By an $\exists \La_R$-\textit{formula} we mean an $\exists_m\La_R$-formula for some $m \in \nat$. If $p = q_1 = 1$ and $r_1 = 0$ (i.e. $\varphi$ consists just of one polynomial equation), we call $\varphi$ a \textit{diophantine $\La_R$-formula}. 

Given a subset $B \subseteq K^n$ and an $\La_R$-formula $\varphi(X_1, \ldots, X_n)$ such that
\begin{displaymath}
B = \lbrace (x_1, \ldots, x_n) \in K^n \mid K \models \varphi(x_1, \ldots, x_n) \rbrace,
\end{displaymath}
we say that $\varphi$ \textit{defines} $B$. A subset of $K^n$ is said to have an \textit{existential} (respectively \textit{diophantine}) \textit{$\La_R$-definition with $m$ quantifiers} if it is equal to the set defined by some existential (respectively diophantine) $\La_R$-formula with $m$ quantifiers. Instead of an existential $\La_R$-definition with $m$ quantifiers, we write $\exists_m\La_R$-\textit{definition} for short. Analogously, we define \textit{universal} $\La_R$-\textit{formulas} and -\textit{definitions} with $m$ quantifiers by replacing the existential quantifiers by universal quantifiers in (\ref{posex}) and we use the notation $\forall_m \La_R$. Note that a subset of $K^n$ has a $\forall_m\La_R$-definition if and only if its complement in $K^n$ has an $\exists_m\La_R$-definition.

It is clear that when $\phi$ and $\psi$ are $\exists_{m_1}\La_R$- and $\exists_{m_2}\La_R$-formulas respectively, then $\phi \vee \psi$ is equivalent to an $\exists_{m}\La_R$-formula for $m = \max\{m_1, m_2\}$ and $\phi \wedge \psi$ is equivalent to an $\exists_{m_1+m_2}\La_R$-formula. This implies in particular that finite intersections and unions of $\exists\La_R$-definable subsets of $K^n$ again have an $\exists\La_R$-definition.

In the rest of this article we will continue to work with general existential formulae instead of diophantine formulae, as the former are more natural to reason with. Furthermore, a well-known statement asserts that for a non-algebraically closed field $K$, every $\exists \La_K$-definable set is also definable by a diophantine $\La_K$-formula. Here is a version of this statement with quantitative bounds and a sketch of a proof for completeness.
\begin{prop}
Let $K$ be a field that is not algebraically closed, $m, n \in \nat$, $B \subseteq K^n$ an $\exists_m \La_K$-definable set. Then $B$ has a diophantine $\La_K$-definition in $K^n$ with $m+1$ quantifiers. 
\end{prop}
\begin{proof}
%Let us say that two formulas with $n$ free variables are \textit{equivalent in $K$} if they define the same subset of $K^n$.
%
%Now, firstly, using that $y_1,\ldots,y_r\in K$ are all nonzero if and only if there exists an element $z\in K$ such that $y_1\cdot\ldots\cdot y_r\cdot z=1$, any finite conjunction of inequalities is equivalent in $K$ to a diophantine formula with one quantifier.
%
%Secondly, as $K$ is not algebraically closed, there exists a nonlinear irreducible form in two variables (e.g. obtained by homogenising a nonlinear irreducible polynomial in one variable), call this form $H(X, Y)$. Then for $x, y \in K$ we have $H(x, y) = 0$ if and only if $x = y = 0$. Using this recursively, it follows that every finite conjunction of equations is equivalent to a single equation.
%
We start from the general formula in (\ref{posex}) and subsequently replace it by formulas defining the same set $B$, eventually ending up with a diophantine formula with $m+1$ quantifiers. Firstly, the observation that for $x_1, \ldots, x_r \in K$ we have $x_1, \ldots, x_r \neq 0$ if and only if $\exists y \in K : x_1\cdots x_ry = 1$ allows us to pass to an equivalent $\exists_{m+1}\La$-formula without inequations, i.e. with $r_i = 0$ for all $i$. Secondly, letting $H(X, Y)$ be the form obtained by homogenising a non-constant univariate polynomial without roots over $K$ -- which exists precisely because $K$ is not algebraically closed -- one can convert a system of two equations into one equation by using that for all $x, y \in K$ one has $x = 0 = y$ if and only if $H(x, y) = 0$. Repeated application of this trick allows us to assume that $q_i = 1$ in (\ref{posex}) for all $i$. Finally, a finite disjunction of equations can be converted into one equation by using that for $x, y \in K$ one has $x = 0$ or $y = 0$ if and only if $xy = 0$, which lets us reduce to $p = 1$.
\end{proof}
\begin{opm}
In the above statement, one can even conclude that $B$ has a diophantine $\La_K$-definition with $m$ quantifiers as soon as $m \geq 1$, see \cite[Corollary 4.12]{DDF}.
\end{opm}
\section{Quaternion algebras}\label{sectquat}
A \textit{quaternion algebra} over $K$ is by definition a $4$-dimensional central simple $K$-algebra. It follows from Wedderburn's Theorem \cite[Corollary 3.5.a]{Pie82} that such an algebra is either a division algebra, in which case we call it \textit{non-split}, or isomorphic to the ring of $2 \times 2$ matrices over $K$, in which case we call it \textit{split}.  When $Q$ is a quaternion algebra over $K$ and $L/K$ is a field extension, then $Q \otimes_K L$ is a quaternion algebra over $L$ \cite[8.5.1]{Sch85}. We denote this algebra by $Q_L$ and say that $Q$ \textit{splits over}~$L$, or that $L$ \textit{splits} $Q$, if $Q_L$ is split.% We refer to \cite[Chapters 12,13]{Pie82} and \cite[Chapter 8]{Sch85} for details and proofs for all claims in this section.

Given $a, b \in K$ with $b(1+4a) \neq 0$, we define the $4$-dimensional $K$-algebra $[a, b)_K = K \oplus Ku \oplus Kv \oplus Kuv$ with $u^2 - u = a$, $v^2 = b$ and $uv + vu = v$. This is a $K$-quaternion algebra and one can show that all quaternion algebras over $K$ are of this form for some~$a$ and $b$. (See \cite[Section IX.10]{Alb39}.)

Given $a, b \in K^\times$, we define the $4$-dimensional $K$-algebra $(a, b)_K = K \oplus Ki \oplus Kj \oplus Kij$ with $i^2 = a, j^2 = b$ and $ij + ji = 0$. If $\charac(K) \neq 2$ this is a $K$-quaternion algebra and one can show that all quaternion algebras over $K$ are of this form for some~$a$ and $b$. Furthermore, one has $[a, b)_K \cong (1+4a, b)_K$ by mapping $v$ to $j$ and $u$ to $\frac{i+1}{2}$. (See \cite[Section IX.10]{Alb39})

We denote by $\Trd$ and $\Nrd$ the \textit{reduced trace map} and the \textit{reduced norm map} on~$Q$ respectively. See \cite[Section 8.5]{Sch85} for basic properties of these two functions. If $Q = [a, b)_K$ and $x = x_1 + x_2u + x_3v + x_4uv$ for some $x_1, \ldots, x_4 \in K$, then one has
\begin{displaymath}
\Trd(x) = 2x_1 + x_2 \quad \text{and} \quad \Nrd(x) = x_1^2 + x_1x_2 - ax_2^2 - b(x_3^2 + x_3x_4 - ax_4^2).
\end{displaymath}
If $Q = (a, b)_K$, $\charac(K) \neq 2$ and $x = x_1 + x_2i + x_3j + x_4ij$ for $x_1, \ldots, x_4 \in K$, then
\begin{displaymath}
\Trd(x) = 2x_1 \quad \text{and} \quad \Nrd(x) = x_1^2 -ax_2^2 - bx_3^2 + abx_4^2.
\end{displaymath}

Here are some results on quaternion algebras that will be used later.
\begin{prop}\label{csa0}
Let $K$ be a field, $a, b \in K$.
\begin{enumerate}
\item Assume $b(1+4a) \neq 0$. The quaternion algebra $[a, b)_K$ is split if and only if $b = x^2 +xy - ay^2$ for some $x, y \in K$.
\item Assume $ab \neq 0$ and $\charac(K) \neq 2$. The quaternion algebra $(a, b)_K$ is split if and only if $b = x^2 - ay^2$ for some $x, y \in K$.
\end{enumerate}
\end{prop}
\begin{proof}
See \cite[Theorem IX.10.26 and Theorem IX.10.27]{Alb39}.
\end{proof}
%\begin{prop}\label{csa1}
%Let $K$ be a field, $Q$ a quaternion algebra over $K$ and let $L/K$ be a quadratic field extension. Then $Q_L$ is split if and only if $L$ can be embedded into $Q$ over $K$.
%\end{prop}
%\begin{proof}
%See \cite[Corollary 13.3]{Pie82}.
%\end{proof}
\begin{prop}\label{csa2}
Let $K$ be a field and $Q$ a quaternion algebra over $K$. Suppose $d \in K$ is such that the splitting field of $X^2 - X - d$ splits $Q$. Then there exists $b \in K$ such that $Q \cong [d, b)_K$.
\end{prop}
\begin{proof}
One can easily see that $\mathbb{M}_2(K) \cong [d, 1)_K$ by considering the matrices
\begin{displaymath}
u = \begin{bmatrix}
0 & d \\ 1 & 1
\end{bmatrix} \quad \text{and} \quad v = \begin{bmatrix}
1 & 1 \\ 0 & -1
\end{bmatrix}
\end{displaymath}
and verifying that they satisfy $u^2 - u = d$, $v^2 = 1$ and $uv + vu = v$. Thus when~$Q$ is split, we can set $b = 1$. If $Q$ is non-split, the result can be derived from a straightforward application of the Skolem-Noether theorem, see e.g. \cite[Proposition 15.1.a]{Pie82}.
\end{proof}
\begin{lem}\label{quatrr}
For any $a, b \in \rr$ with $b \neq 0$, $[a^2, b)_\rr$ is split.
\end{lem}
\begin{proof}
We have $[a^2, b)_\rr \cong (1+4a^2, b)_\rr$ and $1+4a^2$ is a square in $\rr$. Now invoke \Cref{csa0}.
\end{proof}
\section{Local and global fields}\label{sectlocalglobal}
By a \textit{local field} we will mean the fraction field of a complete discrete valuation ring with finite residue field. We call the corresponding $\mathbb{Z}$-valuation on this field the \textit{canonical valuation} of the local field. Note that a finite extension of a local field is again local. The reader is referred to \cite{Eng05} for an overview on valuation theory.

\begin{prop}\label{necessaryconditionsplitlocal}
Let $K$ be a local field. Let $v$ be its canonical valuation and let $F$ be the residue field.
\begin{enumerate}[(a)]
\item If $\charac(F) \neq 2$ and $a, b \in K^\times$ are such that $(a, b)_K$ is a division algebra, then at least one of $v(a)$ and $v(b)$ is odd.
\item If $a, b \in K$ are such that $b(1+4a) \neq 0$ and $[a, b)_K$ is a division algebra, then $v(a) \leq 0$. If additionally $v(a) = 0$, then at least one of $v(1+4a)$ and $v(b)$ is odd.
\end{enumerate}
\end{prop}
\begin{proof}
We denote by $\mathcal{O}$ the valuation ring of $v$ and by $\mathfrak{m}$ its maximal ideal.

Assume $\charac(F) \neq 2$ and that $v(a)$ and $v(b)$ are both even; after multiplying~$a$ and~$b$ with a square -- which does not change the isomorphism class of $(a, b)_K$ -- we may assume $v(a) = v(b) = 0$. Then $a$ and $b$ are non-zero modulo $\mathfrak{m}$, whereby the equation $b = X^2 - aY^2$ has a solution modulo $\mathfrak{m}$ \cite[62:1]{OMe00} and by Hensel's Lemma \cite[Theorem 1.3.1]{Eng05} it then has a solution over $K$, whence $(a, b)_K$ is split by \Cref{csa0}. This shows part (a).

If $\charac(F) \neq 2$, then also $\charac(K) \neq 2$, whereby $[a, b)_K \cong (1+4a, b)_K$ and part (b) follows from part (a): if $v(a) > 0$, then $1+4a$ is a square in $K$ by Hensel's Lemma. Assume for the rest of the proof that $\charac(F) = 2$. %Note that in this case, $v(a) \geq 0$ automatically implies that $v(1+4a) = 0$.

We may multiply $b$ by a square and assume without loss of generality that $v(b) \in \lbrace 0, 1 \rbrace$. If either $v(a) > 0$ or $v(a) = v(b) = 0$, then we can find a non-zero $y \in \mathcal{O}$ such that $a + by^2 \equiv 0 \bmod \mathfrak{m}$, as $F$ is a finite field and of characteristic $2$, whereby every element of $\mathcal{O}$ is a square modulo $\mathfrak{m}$. Then the polynomial $X^2 - X - a - by^2$ has a simple root modulo $\mathfrak{m}$; by Hensel's Lemma it then has a root in $K$, i.e. there exists an $x \in K$ with $0 = x^2 - x - a - by^2$ and thus
\begin{displaymath}
b = \left(\frac{x}{y}\right)^2+ \left(\frac{x}{y}\right)\left(\frac{-1}{y}\right) - a\left(\frac{-1}{y}\right)^2,
\end{displaymath} 
whereby $[a, b)_K$ is split in light of \Cref{csa0}.
\end{proof}
\begin{prop}\label{pitchfork2}
Let $K$ be a local field with canonical discrete valuation ring $\mathcal{O}$ and residue field $F$. Let $d \in \mathcal{O}$ be such that $X^2 - X - d$ is irreducible over the residue field. Then for all $b \in K^\times$ we have that $[d, b)_K$ is split if and only if $v(b)$ is even.
\end{prop}
\begin{proof}
Note that $v(1 + 4d) = 0$, otherwise $X^2 - X - d$ would be reducible over the residue field. Hence, if $v(b)$ is even, it follows from \Cref{necessaryconditionsplitlocal} that $[d, b)_K$ is split.

The form $X^2 - XY - Y^2d$ has no non-trivial zeroes over $F$. Thus the form can only represent elements of $K$ of even value under $v$. Hence if $v(b)$ is odd, there cannot be $x, y \in K$ with $b = x^2 + xy - y^2d$, whereby $[d, b)_K$ is non-split in light of \Cref{csa0}.
\end{proof}
\begin{prop}\label{classlocal}
Let $K$ be a local field. For every quadratic field extension $L/K$ and any quaternion algebra $Q$ over $K$, $Q_L$ is split.
\end{prop}
\begin{proof}
See \cite[Section 17.10]{Pie82}.
\end{proof}

We call a finite extension of $\mathbb{Q}$ a \textit{number field} and a function field in one variable over a finite field a \textit{global function field}. By a \textit{global field} we mean either a number field or a global function field.

Given a field $K$, we consider the set of $\mathbb{Z}$-valuations on $K$, which we denote by $\mathcal{V}_K$. Given $v \in \mathcal{V}_K$, we may denote by $\mathcal{O}_v$ the valuation ring of $K$ corresponding to the valuation $v$ and by $\mathfrak{m}_v$ the maximal ideal of $\mathcal{O}_v$. We write $K_{v}$ for the fraction field of the completion of $\mathcal{O}_v$.

Note that if $K$ is a global field, then the completion of any discrete valuation ring has a fraction field which is a local field. Furthermore, for all $x \in K^\times$ there exist only finitely many $v \in \mathcal{V}_K$ with $v(x) \neq 0$.

Given a quaternion algebra $Q$ defined over a field $K$, we call $Q$ \textit{nonreal} if $Q$ is split over every real closure of $K$. By definition, if $K$ does not have real closures, all quaternion algebras over $Q$ are nonreal. A quaternion algebra $Q$ over a global field is nonreal if and only if $Q$ is split over every embedding into $\rr$.

Let $Q$ be a quaternion algebra over $K$. We define the \textit{ramification set} of $Q$ as
\begin{displaymath}
\Delta Q = \lbrace v \in \mathcal{V}_K \mid Q_{K_v} \text{ is not split} \rbrace.
\end{displaymath}
The sets $\Delta_{a, b}$ from the introduction (for $a, b \in \qq^\times$) correspond to $\Delta((a, b)_\qq)$ in this new notation.

\begin{stel}[Albert-Brauer-Hasse-Noether]\label{HasseMinkowski}
Let $K$ be a global field and $Q$ a nonreal quaternion algebra over $K$. Then $\Delta Q = \emptyset$ if and only if $Q$ is split.
\end{stel}
\begin{proof}
See \cite[Theorem 8.1.17]{Neu15}.
\end{proof}
\begin{stel}[Hilbert Reciprocity]\label{Hilbertreciprocity}
Let $K$ be a global field. If $Q$ is a nonreal quaternion algebra over $K$, then $\lvert \Delta Q \rvert \in 2\mathbb{N}$. Conversely, given any subset $S \subseteq \mathcal{V}_K$ with $\lvert S \rvert \in 2\mathbb{N}$, there exists a nonreal quaternion algebra $Q$ over $K$ with $\Delta Q = S$.
\end{stel}
\begin{proof}
See \cite[Theorem 8.1.17]{Neu15}.
\end{proof}
We can derive a more explicit form of the second part of the last theorem.
\begin{gev}\label{globalclassification}
Let $K$ be a global field. Let $S \subseteq \mathcal{V}_K$ be such that $\lvert S \rvert \in 2\mathbb{N}$. Let $d \in K$ be such that $1 + 4d \neq 0$ and suppose that for all $v \in S$, the polynomial $X^2 - X - d$ is irreducible over $K_{v}$. Then there exists $b \in K^\times$ such that $\Delta([d, b)_K) = S$ and $[d, b)_K$ is nonreal.
\end{gev}
\begin{proof}
By \Cref{Hilbertreciprocity} there exists a nonreal $K$-quaternion algebra $Q$ such that $\Delta Q = S$. Let $L$ be the splitting field of $X^2 - X -d$ over $K$. Clearly $Q_L$ is again nonreal. We show that $\Delta(Q_L) = \emptyset$; this implies via \Cref{HasseMinkowski} that $Q_L$ splits, and then the statement follows from \Cref{csa2}.

Consider a $\mathbb{Z}$-valuation $w$ of $L$. Then $L_{w} \cong LK_{v}$ for some $\mathbb{Z}$-valuation $v$ of $K$. If $v \not\in S$, then by the choice of $Q$, $Q_{K_v}$ is split, so also $Q_{L_w}$ is split. On the other hand, if $v \in S$, $L_{w}/K_{v}$ is a quadratic extension by the assumption that $X^2 - X - d$ is irreducible over $K_{\mathfrak{p}}$. \Cref{classlocal} then implies that $L_{w}$ splits $Q$. We conclude that, for all $\mathbb{Z}$-valuations $w$ of $L$, $Q_{L_w}$ is split. This shows that $\Delta(Q_L) = \emptyset$.
\end{proof}

For later use, we also mention an instance of the classical Weak Approximation Theorem from valuation theory.
\begin{stel}[Weak Approximation]\label{WAT}
Let $K$ be a field, $v_1, \ldots, v_n$ pairwise distinct $\zz$-valuations on $K$. Then for any $a_1, \ldots, a_n \in K$ and $\gamma_1, \ldots, \gamma_n \in \zz$ there exists an $x \in K$ with
\begin{displaymath}
v_i(x - a_i) > \gamma_i \quad \text{for all } i \in \lbrace 1,\ldots, n \rbrace.
\end{displaymath}
\end{stel}
\begin{proof}
See \cite[Theorem 2.4.1]{Eng05}; using that distinct $\mathbb{Z}$-valuations are trivially independent by \cite[Corollary 2.3.2]{Eng05}.
\end{proof}
\section{Semilocal subrings}\label{sectsemilocal}
Let $Q, Q'$ be quaternion algebras over $K$ and $L/K$ a quadratic field extension. Consider the following subsets of $K$:
\begin{align*}
S(Q) &= \lbrace \Trd(x) \mid x \in Q \setminus K, \Nrd(x) = 1 \rbrace, \\
T(Q, Q') &= \bigcap \lbrace \mathcal{O}_v \mid v \in \Delta(Q) \cap \Delta(Q') \rbrace.
\end{align*}
We will write $T(Q)$ instead of $T(Q, Q)$. 
We will see that these sets have good $\exists\La_K$-definitions (see \Cref{TIJquantifierslem}). They were introduced in the context of defining subrings of global fields by Poonen in \cite{Poo09} and Koenigsmann in \cite{Koe16}.

Throughout this article, when dealing with subsets of a field $K$, we take the convention that $\bigcap \emptyset = K$.
\begin{stel}\label{AEprop}
Let $Q, Q'$ be nonreal quaternion algebras over a global field $K$. Then
\begin{displaymath}
T(Q,Q') =  S(Q) + S(Q').
\end{displaymath}
\end{stel}
\begin{proof}
%This is a corollary of \Cref{HasseMinkowski}.
In \cite[Proposition 2.9]{Dit17} a proof is given in the case $Q = Q'$; by inspection one sees how the proof can be easily modified to cover the general case.
\end{proof}

For a subset $A$ of a field $K$, denote $A^{-1} = \lbrace x \in K^\times \mid x^{-1} \in A \rbrace$ and $A^\times = A \cap A^{-1}$.
\begin{lem}\label{TIJquantifierslem}
There exist $\exists_7\La$-formulas $\varphi_1, \varphi_2, \varphi_3$ in the variables $(X, A, B, A', B')$ such that, for all global fields $K$ and for all $a, b, a', b' \in K$ with $(1+4a)b'(1+4a')b' \neq~0$ and such that $[a, b)_K$ and $[a', b')_K$ are nonreal, we have
\begin{align*}
T([a, b)_K, [a', b')_K) &= \lbrace x \in K \mid K \models \varphi_1(x, a, b, a', b') \rbrace, \\
T([a, b)_K, [a', b')_K)^{-1} &= \lbrace x \in K \mid K \models \varphi_2(x, a, b, a', b') \rbrace, \\
T([a, b)_K, [a', b')_K)^{\times} &= \lbrace x \in K \mid K \models \varphi_3(x, a, b, a', b') \rbrace.
\end{align*}
\end{lem}
\begin{proof}
By the formulas for reduced trace and norm given in the third section, we have for $a, b \in K$ with $(1+4a)b \neq 0$ that
\begin{displaymath}
%S((a, b)_K) &= \lbrace t \in K \mid \exists x_2, x_3, x_4 \in K : (t^2 - 4ax_2^2 - 4bx_3^2 + 4abx_4^2 = 4) \rbrace, \\
S([a, b)_K) = \left\lbrace t \in K \middle| \begin{array}{l}
\exists x_1, x_3, x_4 \in K : (x_1^2 + x_1(t - 2x_1) - a(t - 2x_1)^2 \\ - b(x_3^2 + x_3x_4 - ax_4^2) = 1)
\end{array} \right\rbrace.
\end{displaymath}
Thus there is an $\exists_3\La$-formula taking $a$ and $b$ as parameters defining $S([a, b)_K)$. As such, there is an $\exists_7\La$-formula $\varphi_1$ defining $T([a, b)_K, [a', b')_K)$, since for $x \in K$ one has
\begin{displaymath}
x \in T([a, b)_K, [a', b')_K) \Leftrightarrow \exists y \in K : (y \in S([a, b)_K) \text{ and } x - y \in S([a', b')_K)).
\end{displaymath}
To find $\varphi_2$, one expresses that $x \neq 0$, subtitutes $\frac{1}{x}$ for $x$ into $\varphi_1$ and then clears denominators. For $\varphi_3$, it suffices to express that $x \neq 0$, substitute $\frac{x^2 + 1}{x}$ for $x$ in $\varphi_1$ and clear denominators. Indeed we have for nonreal quaternion algebras $Q = [a, b)_K$ and $Q' = [a', b')_K$ that
\begin{displaymath}
T(Q,Q')^\times = \bigcap \lbrace \mathcal{O}_v^\times \mid v \in \Delta(Q) \cap \Delta(Q') \rbrace.
\end{displaymath}
For a valuation $v \in \mathcal{V}_K$ and $x \in K^\times$, it is easy to see that $x \in \mathcal{O}_v^\times$ if and only if $\frac{x^2 + 1}{x} \in \mathcal{O}_v$. Hence $x \in T(Q, Q')^\times$ if and only if $\frac{x^2 + 1}{x} \in T(Q, Q')$.
\end{proof}
\begin{prop}\label{AEpropgev2}
Let $K$ be a global field. Let $S$ be a non-empty, finite subset of~$\mathcal{V}_K$. Then the subsets $\bigcap_{v \in S} \mathcal{O}_v$ and $\bigcap_{v \in S} \mathcal{O}_v^\times$ have $\exists_7\La_K$-definitions.
\end{prop}
\begin{proof}
We can find two finite sets $S_1,S_2 \subseteq \mathcal{V}_K$ of even cardinality such that $S_1 \cap S_2 = S$; by \Cref{Hilbertreciprocity} we can find nonreal $K$-quaternion algebras $Q_1$ and $Q_2$ such that $\Delta(Q_1) = S_1$ and $\Delta(Q_2)=S_2$. We obtain that $\bigcap_{v\in S} \mathcal{O}_v = T(Q_1, Q_2)$; by \Cref{TIJquantifierslem} this set and $T(Q_1, Q_2)^\times$ have an $\exists_7\La_K$-definition.
\end{proof}
For $c \in K^\times$, we define the following finite subsets of $\mathcal{V}_K$:
\begin{align*}
\Odd(c) &= \lbrace v \in \mathcal{V}_K \mid v(c) \text{ is odd} \rbrace \\
\Neg(c) &= \lbrace v \in \mathcal{V}_K \mid v(c) < 0 \rbrace.
\end{align*}

For a quaternion algebra $Q$ over $K$ and an element $c \in K^\times$ we define the following subsets of $K$:
\begin{align*}
\square K &= \lbrace x \in K \mid \exists y \in K : x = y^2 \rbrace, \\
%I^c(Q) &= c \cdot \square K \cdot T(Q)^\times \cap (1 - \square K \cdot T(Q)^\times) \\
J^c(Q) &= \bigcap \lbrace \mathfrak{m}_v \mid v \in \Delta Q \cap \Odd(c) \rbrace, \\
H^c(Q) &= \bigcap \lbrace \mathfrak{m}_v^{-v(c)} \mid v \in \Delta Q \cap \Neg(c) \rbrace.
\end{align*}
The sets $J^c_{a, b}$ from the introduction (for $a, b, c \in \qq^\times$) correspond to $J^c((a, b)_\qq)$ in this notation.
The sets $H^c(Q)$ will in the end only play a role when considering global fields of characteristic 2, but in this section we make no assumptions on the characteristic.
%These are diophantine subsets of $K$; we count their rank in \Cref{TIJquantifiers}. 
\begin{lem}\label{jacoblem}
Let $K$ be a field, $S \subseteq \mathcal{V}_K$ finite , $R = \bigcap_{v \in S} \mathcal{O}_v$. For $c \in K^\times$ we have
\begin{align*}
(c \cdot \square K \cap (1 - \square K \cdot R^\times))\cdot R &= \bigcap \lbrace\mathfrak{m}_v \mid v \in S \cap \Odd(c) \rbrace, \\
(c^{-1} \cdot R + c \cdot R^{-1})^{-1} \cup \lbrace 0 \rbrace &= \bigcap\lbrace\mathfrak{m}_v^{-v(c)} \mid v \in S \cap \Neg(c) \rbrace.
\end{align*}
\begin{proof}
First we observe that, by Weak Approximation, we have
\begin{displaymath}
\square K \cdot R^\times = \bigcap_{v \in S} v^{-1}(2\mathbb{Z}).
\end{displaymath}
Let $x \in K^\times$ be arbitrary. Let us consider $v \in S \cap \Odd(c)$. If $x \in c \cdot \square K$ then $v(x)$ is odd, and if additionally $x \in 1 - \square K \cdot R^\times$, then it follows that $v(x)$ is strictly positive. Therefore if $x \in c \cdot \square K \cap (1 - \square K \cdot R^\times)$, then $x \in \bigcap \lbrace\mathfrak{m}_v \mid v \in S \cap \Odd(c) \rbrace$, whereby also $Rx \subseteq \bigcap \lbrace\mathfrak{m}_v \mid v \in S \cap \Odd(c) \rbrace$. This proves the left-to-right inclusion in the first equality.

For the other inclusion, let us consider $x \in \bigcap \lbrace\mathfrak{m}_v \mid v \in S \cap \Odd(c) \rbrace$. By Weak Approximation, there exists $z \in K$ such that for all $v \in S \cap \Odd(c)$ we have $v(cz^2) = 1$ and for all $v \in S \setminus \Odd(c)$ we have $v(cz^2) < \min \lbrace 0, v(x) \rbrace$. Then $cz^2 \in (c \cdot \square K \cap (1 - \square K \cdot R^\times))$ and $v(cz^2) < v(x)$ for all $v \in S$, whereby $\frac{x}{cz^2} \in R$. Thus we have $x \in (c \cdot \square K \cap (1 - \square K \cdot R^\times))\cdot R$.

We now give a proof of the second equality. Let $x \in (c^{-1}R + cR^{-1})^{-1}$. Then $x^{-1} = c^{-1}t' + ct^{-1}$ for some $t' \in R$, $t \in R \setminus \lbrace 0 \rbrace$. For $v \in S \cap \Neg(c)$, we have $R \subseteq \mathcal{O}_v$, $v(c^{-1}t') = - v(c) + v(t') > 0$ and $v(ct^{-1}) = v(c) - v(t) < 0$, hence
\begin{displaymath}
v(x^{-1}) = \min \lbrace v(c^{-1}t'), v(ct^{-1}) \rbrace = v(c) - v(t) \leq v(c),
\end{displaymath}
whereby $v(x) \geq -v(c)$. This shows the left-to-right inclusion of the second equality in the statement.

Now let $x \in \bigcap \lbrace \mathfrak{m}^{-v(c)} \mid v \in S \cap \Neg(c) \rbrace$ with $x \neq 0$. Then $v(x) \geq -v(c) > 0$ for all $v \in S \cap \Neg(c)$. We will show that for any $v \in S$ we can find $t_{v}, t'_{v} \in \mathcal{O}_v$ with $t_{v} \neq 0$ such that $t_{v}' = cx^{-1} - c^2t_{v}^{-1}$. Once this is shown, it follows by Weak Approximation that there exist $t, t' \in \bigcap_{v \in S}\mathcal{O}_v = R$, $t \neq 0$ such that $t' = cx^{-1} - c^2t^{-1}$, whereby we will have that $x = (c^{-1}t' + ct^{-1})^{-1}$, as we want to show.

Let us consider $v \in S$. Assume first that $v(x) \geq -v(c)$. In this case we take $t_{v} = xc$ and $t'_v = 0$. Now suppose that $v(x) < -v(c)$. By our assumption on $x$ this is only possible when $v(c) \geq 0$. In this case we take $t_{v} = 1$ and $t'_v = c(x^{-1} - c)$.
\end{proof}
\end{lem}
\begin{prop}\label{TIJquantifiers}
There exist an $\exists_{16}\La$-formula $\phi_1(X, A, B)$ and an $\exists_{15}\La$-formula $\phi_2$ in the variables $(X, A, B)$ such that, for all global fields $K$ and $a, b, c \in K$ with $(1+4a)bc \neq 0$ such that $[a, b)_K$ is nonreal, we have
\begin{align*}
J^c([a, b)_K) &= \lbrace x \in K \mid K \models \phi_1(x, a, b, c) \rbrace, \\
H^c([a, b)_K) &= \lbrace x \in K \mid K \models \phi_2(x, a, b, c) \rbrace.
\end{align*}
\end{prop}
\begin{proof}
By \Cref{jacoblem} we have that
\begin{align*}
J^c([a, b)_K) &= (c \cdot \square K \cap (1 - \square K \cdot T([a, b)_K)^\times))\cdot T([a, b)_K), \\
H^c([a, b)_K) &= (c^{-1} \cdot T([a, b)_K) + c \cdot T([a, b)_K)^{-1})^{-1} \cup \lbrace 0 \rbrace.
\end{align*}
For $x \in K$ we have that $x \in \square K \cdot T([a, b)_K)^\times$ if and only if
\begin{displaymath}
\exists q \in K^\times : xq^2 \in T([a, b)_K)^\times \cup \lbrace 0 \rbrace
\end{displaymath}
which by \Cref{TIJquantifierslem} can be described with an $\exists_8\La$-formula.

For $x \in K$ we have that $x \in J^c([a, b)_K)$ if and only if
\begin{displaymath}
\exists y \in K^\times : \frac{x}{cy^2} \in T([a, b)_K) \text{ and } 1 - cy^2 \in \square K \cdot T([a, b)_K)^\times.
\end{displaymath}
Hence, invoking \Cref{TIJquantifierslem} and clearing denominators to express that $\frac{x}{cy^2} \in T([a, b)_K)$, we find an $\exists_{16}\La$-formula for $\phi_1$.

To see that $H^c(Q_{a,b})$ can be described with an $\exists_{15}\La$-formula, note that for $x \in K^\times$ we have that $x \in H^c([a, b)_K)$ if and only if
\begin{displaymath}
\exists y \in K^\times : cy \in T([a, b)_K) \text{ and } \frac{1-xy}{cx} \in T([a, b)_K)^{-1},
\end{displaymath}
and then invoke \Cref{TIJquantifierslem} and clear denominators.
\end{proof}

\section{Rings of $S$-integers}\label{sectS}
For the remainder of this article, let $K$ be a global field. For a set $S \subseteq \mathcal{V}_K$, we define the set
\begin{displaymath}
\mathcal{O}_S = \lbrace x \in K \mid \forall v \in \mathcal{V}_K \setminus S : v(x) \geq 0 \rbrace = \bigcap_{v \in \mathcal{V}_K \setminus S} \mathcal{O}_v.
\end{displaymath}
If $S$ is a finite set, we call $\mathcal{O}_S$ \textit{the ring of $S$-integers}.
\begin{prop}\label{EtoA}
Let $V \subseteq \mathcal{V}_K$ be non-empty. Suppose that the set $\bigcup_{v \in V} \mathfrak{m}_v$
has a $\exists_n\La_K$-definition. Then $\bigcap_{v \in V} \mathcal{O}_v$ has a $\forall_{n}\La_K$-definition.
\end{prop}
\begin{proof}
This follows from the observation that
\begin{displaymath}
\bigcap_{v \in V} \mathcal{O}_v = \left(K \setminus \left(\bigcup_{v \in V} \mathfrak{m}_v\right)^{-1}\right) \cup \lbrace 0 \rbrace.
\end{displaymath}
\end{proof}
\begin{gev}\label{EtoAgev}
Let $S \subseteq \mathcal{V}_K$ be finite and non-empty. Then $\bigcup_{v \in S} \mathfrak{m}_v$ has an $\exists_7\La_K$-definition and $\bigcap_{v \in S} \mathcal{O}_v$ has a $\forall_7\La_K$-definition in $K$.
\end{gev}
\begin{proof}
The second statement follows from the first one by \Cref{EtoA}. By \Cref{AEpropgev2}, for every $v \in \mathcal{V}_K$, $\mathcal{O}_v$ has an $\exists_7\La_K$-definition. Then the same holds for $\mathfrak{m}_v$: fix an arbitrary $\pi \in \mathfrak{m}_v \setminus \mathfrak{m}^2_{(v)}$; we have $\mathfrak{m}_v = \pi\mathcal{O}_v$, whence $x \in \mathfrak{m}_v$ if and only if $\frac{x}{\pi} \in \mathcal{O}_v$.
If every $\mathfrak{m}_v$ has an $\exists_7\La_K$-definition, then so does a finite union of such sets.
\end{proof}
We will show that for any finite set $S \subseteq \mathcal{V}_K$, there is an $\exists\La_K$-definition of
\begin{displaymath}
\bigcup_{v \in \mathcal{V}_K \setminus S } \mathfrak{m}_v
\end{displaymath}
in $K$. From this we will obtain a $\forall\La_K$-definition of $\mathcal{O}_{S}$ via \Cref{EtoA}.

In particular, setting $S = \emptyset$, we find a universal definition of the ring of integers $\mathcal{O}_K$ in a number field $K$. However, even if one is only interested in the case $S = \emptyset$, it will be crucial in the proof to also allow $S$ to be non-empty.

\begin{lem}\label{reduceToOdd}
Let $V \subseteq V' \subseteq \mathcal{V}_K$ and suppose that $V' \setminus V$ is finite. Assume that $\bigcup_{v \in \mathcal{V}_K \setminus V'} \mathfrak{m}_v$ has an $\exists_n\La_K$-definition. Then the set $\bigcup_{v \in \mathcal{V}_K \setminus V} \mathfrak{m}_v$ has an $\exists_{m}\La_K$-definition with $m = \max\{ n, 7 \}$.
\end{lem}
\begin{proof}
Since we know from \Cref{EtoAgev} that $\bigcup_{v \in V' \setminus V} \mathfrak{m}_v$ has an $\exists_7\La_K$-definition, we obtain that the set $\bigcup_{v \in \mathcal{V}_K \setminus V} \mathfrak{m}_v$ has an $\exists_{\max\{ n, 7 \}}\La_K$-definition by observing that
\begin{displaymath}
\bigcup_{v \in \mathcal{V}_K \setminus V} \mathfrak{m}_v = \bigcup_{v \in \mathcal{V}_K \setminus V'} \mathfrak{m}_v \cup \bigcup_{v \in V' \setminus V} \mathfrak{m}_v.
\end{displaymath}
\end{proof}
In particular, to prove that $\bigcup_{v \in \mathcal{V}_K \setminus S} \mathfrak{m}_v$ has an $\exists\La_K$-definition for all finite sets $S$, it is enough to show this for sufficiently large sets $S$ of finite cardinality.

Let $S \subseteq \mathcal{V}_K$ be a non-empty finite set, and $u \in \bigcap_{v \in S} \mathcal{O}_v^\times$. Define the set
\begin{displaymath}
\Phi^S_u = \left\lbrace (a, b) \in K^2 \enspace\middle|\enspace b \in \bigcap_{v \in S}\mathcal{O}_v^\times, a \equiv u \bmod \prod_{v \in S}\mathfrak{m}_v \right\rbrace.
\end{displaymath}
\begin{lem}\label{nicomainthmlem}
Let $S \subseteq \mathcal{V}_K$ be a non-empty finite set and $u \in \bigcap_{v \in S} \mathcal{O}_v^\times$. Then the set $\Phi^S_u$ has a $\exists_{14}\La_K$-definition in $K^2$.
\end{lem}
\begin{proof}
By \Cref{AEpropgev2} $\bigcap_{v \in S} \mathcal{O}_v^\times$ has a $\exists_7\La_K$-definition. By Weak Approximation, fix an element $\pi \in K^\times$ with $\pi \in \mathfrak{m}_v \setminus \mathfrak{m}_v^2$ for all $v \in S$. The condition that $a \equiv u \bmod \prod_{v \in S}\mathfrak{m}_v$ can be rewritten as $a-u \in \pi\bigcap_{v \in S} \mathcal{O}_v$. By \Cref{AEpropgev2} this can be described by an $\exists_7\La_K$-formula. This brings the total to $7+7=14$ quantifiers.
\end{proof}
\begin{lem}\label{pitchfork1}
Let $F$ be a finite field. There exists $u \in F$ such that $X^2 - X - u^2$ is irreducible over $F$.
\end{lem}
\begin{proof}
If $\charac(F) = 2$, then every element is a square and the statement becomes trivial, as $F$ has a separable quadratic extension. Suppose now that $\charac(F) \neq 2$; we need to show that there exists $u \in F$ such that the discriminant $1 + 4u^2$ is not a square. To this end, take any $b \in F$ which is not a square. Using that every element is a sum of two squares in $F$ \cite[62:1]{OMe00}, write $b = c^2 + d^2$ for some $c, d \in F$. Then $u = \frac{d}{2c}$ does the trick, as $1+4u^2 = \frac{b}{c^2}$ is not a square.
\end{proof}

\begin{lem}\label{nicomainthm}
Let $K$ be a global field. Let $\pi \in K^\times$ be such that $S = \Odd(\pi)$ has odd cardinality. Let $u \in K^\times$ be such that for all $v\in S$, $v(u) = 0$ and $X^2 - X - u^2$ is irreducible over $\mathcal{O}_v/\mathfrak{m}_v$.

If $\charac(K) = 2$, then we have
\begin{displaymath}
\bigcup_{v \in \mathcal{V}_K \setminus S } \mathfrak{m}_v = \bigcup_{(a, b) \in \Phi_u^S} (J^b([a^2, b\pi)_K) \cap H^a([a^2, b\pi)_K)).
\end{displaymath}
If $\charac(K) \neq 2$ and $S$ contains all (finitely many) valuations $v \in \mathcal{V}_K$ with $v(2) > 0$, then we have
\begin{displaymath}
\bigcup_{v \in \mathcal{V}_K \setminus S } \mathfrak{m}_v = \bigcup_{(a, b) \in \Phi_u^S} (J^{1+4a^2}([a^2, b\pi)_K) \cap J^b([a^2, b\pi)_K)).
\end{displaymath}
\end{lem}
\begin{proof}
We start by showing the right-to-left inclusion in both cases. Take an arbitrary $(a, b) \in \Phi_u^S$. By the definition of $\Phi_u^S$ we have for all $v \in S$ that $v(a) = 0$, $v(b\pi) = v(b) + v(\pi) \equiv 0 + 1 \bmod 2$ and $X^2 - X - a^2$ is irreducible over $\mathcal{O}_v/\mathfrak{m}_v$. It follows by \Cref{pitchfork2} that $S \subseteq \Delta[a^2, b\pi)_K$. By \Cref{quatrr} $[a^2, b\pi)_K$ is nonreal. As $\lvert S \rvert$ is odd, Hilbert Reciprocity (\Cref{Hilbertreciprocity}) tells us that there exists $w \in \Delta[a^2, b\pi)_K \setminus S$. By part (b) of \Cref{necessaryconditionsplitlocal} at least one of the following holds:
\begin{enumerate}[(i)]
\item\label{eq:caseChar2} $2w(a) = w(a^2) < 0$. In this case, $H^a([a^2, b\pi)_K) \subseteq \mathfrak{m}_{w}$.
\item\label{eq:caseCharneq2} $w(1+4a^2)$ is odd. In this case, $J^{1+4a^2}([a^2, b\pi)_K) \subseteq \mathfrak{m}_{w}$.
\item\label{eq:casebPiOdd} $w(b\pi)$ is odd, whereby $w(b)$ is odd (since $w \not\in S = \Odd(\pi)$) and thus $J^{b}([a^2, b\pi)_K) \subseteq \mathfrak{m}_{w}$.
\end{enumerate}
Note that case \eqref{eq:caseCharneq2} does not occur if $\charac(K) = 2$. If case \eqref{eq:caseChar2} occurs and $w(2) =~0$, then by part (a) of \Cref{necessaryconditionsplitlocal} also either \eqref{eq:caseCharneq2} or \eqref{eq:casebPiOdd} occurs.
We conclude that $J^{b}([a^2, b\pi)_K) \cap H^a([a^2, b\pi)_K) \subseteq \mathfrak{m}_{w}$ if $\charac(K) = 2$, and $J^{b}([a^2, b\pi)_K) \cap J^{1+4a^2}([a^2, b\pi)_K) \subseteq \mathfrak{m}_w$ if $w(2) = 0$. As this argument works for general $(a, b) \in \Phi_u^S$, this shows the right-to-left inclusion in each of the two cases.

To show the inclusion from left to right in both cases, it suffices to show that for any given $w \in \mathcal{V}_K \setminus S$ there exist $(a, b) \in \Phi_u^S$ such that $\Delta[a^2, b\pi)_K = S \cup \lbrace w \rbrace$. Indeed, having found such $(a, b)$ one has for all $v \in S$ that $v(1+4a^2) = v(a) = v(b) = 0$, from which it follows that $\mathfrak{m}_{w} \subseteq J^{1+4a^2}([a^2, b\pi)_K) \cap J^b([a^2, b\pi)_K) \cap H^a([a^2, b\pi)_K)$.

Given $w \in \mathcal{V}_K \setminus S$, by \Cref{pitchfork1} and Weak Approximation there exists $a \in K^\times$ such that $a \equiv u \bmod \prod_{v \in S}\mathfrak{m}_v$, $w(a) = 0$ and $X^2 - X - a^2$ is irreducible over $\mathcal{O}_{(w)}/\mathfrak{m}_{(w)}$. By \Cref{globalclassification} we can find $b \in K^\times$ such that $\Delta[a^2, b\pi)_K = S \cup \lbrace w \rbrace$. \Cref{pitchfork2} tells us that $v(b\pi) = v(b) + v(\pi)$ is odd for all $v \in S$, whereby $v(b)$ is even. Hence by Weak Approximation we may multiply $b$ by an appropriate square and assume without loss of generality that $b \in \bigcap_{v \in S} \mathcal{O}_v^\times$. Then $(a, b) \in \Phi_u^S$ and $\Delta[a^2, b\pi)_K = S \cup \lbrace w \rbrace$, whereby we are done.
\end{proof}

\begin{lem}\label{Dirichlet}
Let $K$ be a global field, $S \subseteq \mathcal{V}_K$ finite.
There exists $\pi \in K^\times$ such that $S \subseteq \Odd(\pi)$ and $\Odd(\pi)$ has odd cardinality.
\end{lem}
\begin{proof}
If needed, we enlarge $S$ to be a set of even cardinality. It follows from results of class field theory - more specifically a generalisation of Dirichlet's theorem on arithmetic progressions as formulated e.g. in \cite[A.10]{BMS} - that there exist infinitely many $v \in \mathcal{V}_K$ such that $S \cup \lbrace v \rbrace = \Odd(\pi)$ for some $\pi \in K^\times$.
\end{proof}

\begin{stel}\label{nicomainthmgev}
Let $K$ be a global field and $S$ a finite subset of $\mathcal{V}_K$. The set $\bigcup_{v \in \mathcal{V}_K \setminus S } \mathfrak{m}_v$ has an $\exists\La_K$-definition in $K$. Furthermore, $\mathcal{O}_S$ has a $\forall\La_K$-definition in $K$.
\end{stel}
\begin{proof}
By \Cref{Dirichlet} there exists $\pi \in K^\times$ such that $S \subseteq S' = \Odd(\pi)$ and $S'$ has odd cardinality.
By Weak Approximation and \Cref{pitchfork1} we can find $u \in K$ such that for all $v \in S'$ one has $v(u) = 0$ and $X^2 - X - u^2$ irreducible over $\mathcal{O}_v/\mathfrak{m}_v$.

By \Cref{TIJquantifiers} and \Cref{nicomainthmlem}, the sets $\bigcup\lbrace J^b([a^2, b\pi)_K) \cap H^a([a^2, b\pi)_K) \mid (a, b) \in \Phi_u^{S'} \rbrace$ and $\bigcup\lbrace J^{1+4a^2}([a^2, b\pi)_K) \cap J^b([a^2, b\pi)_K) \mid (a, b) \in \Phi_u^{S'} \rbrace$ have $\exists\La_K$-definitions, hence by \Cref{nicomainthm} one has that $\bigcup_{v \in \mathcal{V}_K \setminus S' } \mathfrak{m}_v$ has an $\exists\La_K$-definition. By \Cref{reduceToOdd} we know that also $\bigcup_{v \in \mathcal{V}_K \setminus S} \mathfrak{m}_v$ has an $\exists\La_K$-definition. The second part now follows from \Cref{EtoA}.
\end{proof}
\begin{opm}
A variation of \Cref{nicomainthm} can be proven which does not require invoking additional results from class field theory (as in \Cref{Dirichlet}) to find an appropriate element $\pi$, but instead requires introducing another auxiliary parameter. This approach can be found in an earlier preprint of this article, available via \url{https://arxiv.org/abs/1812.04372v4} (Lemma 6.6).
\end{opm}
\section{Quantifier count and optimisation}\label{sectquantcount}
As explained in the proof of \Cref{nicomainthmgev}, the equality in \Cref{nicomainthm} gives rise to an $\exists\La_K$-definition of $\bigcup_{v \in \mathcal{V}_K \setminus S } \mathfrak{m}_v$, which then leads to a $\forall\La_K$-definition of $\mathcal{O}_S$ via \Cref{EtoA}. We can count the number of quantifiers in the definition obtained:
\begin{enumerate}[1.]
 \item By \Cref{TIJquantifiers} the sets $J^d(Q_{a^2,b})$ (for $d = 1+4a^2, b$) and $H^a(Q_{a^2,b})$ can be defined with $\exists\La$-formulas with 16 and 15 quantifiers, respectively. This implies that the intersections
\begin{displaymath}
J^b([a^2, b\pi)_K) \cap J^{1+4a^2}([a^2, b\pi)_K) \quad \text{and} \quad J^b([a^2, b\pi)_K) \cap H^a([a^2, b\pi)_K)
\end{displaymath}
can be defined with $16 + 16 = 32$ and $16 + 15 = 31$ quantifiers respectively.
\item By \Cref{nicomainthmlem} the set $\Phi_u^S$ needs $14$ quantifiers to define.
\item In total this yields a definition for $$\bigcup_{v \in \mathcal{V}_K \setminus S } \mathfrak{m}_v$$ with $2 + 14 + 32 = 48$ quantifiers if $\charac(K) \neq 2$, or $2 + 14 + 31 = 47$ quantifiers if $\charac(K) = 2$.
\end{enumerate}
This bound can be improved, using a result for reducing the number of quantifiers needed when taking a conjunction of existential formulas.
\begin{stel}\label{T:n+m-1}
Let $K$ be a field which is finitely generated over a perfect field.
For any $m_1, m_2, n \in \nat$ with $m_1, m_2 \geq 1$ and subsets $D_1, D_2 \subseteq K^n$ which are $\exists_{m_1}\La_K$-definable and $\exists_{m_2}\La_K$-definable respectively, the intersection $D_1 \cap D_2$ is $\exists_{m_1+m_2-1}\La_K$-definable.
\end{stel}
\begin{proof}
This is \cite[Theorem 1.4]{DDF}. In \cite[Theorem 1.2]{ZhangSunZinQ} a special case of this result is shown for $K = \qq$ for the same purposes as we use it, yielding a very explicit formula.
\end{proof}
\begin{prop}
The set $\bigcup_{v \in \mathcal{V}_K \setminus S } \mathfrak{m}_v$ has an $\exists_{37}\La_K$-definition in $K$. When $\charac(K) = 2$, it also has an $\exists_{36}\La_K$-definition in $K$.
\end{prop}
\begin{proof}
This follows by going through the quantifier count at the beginning of the section and applying \Cref{T:n+m-1} every time a conjunction of existential formulas is taken in the construction. This means that consecutively \Cref{TIJquantifierslem}, \Cref{AEpropgev2}, \Cref{TIJquantifiers} and \Cref{nicomainthmlem} need to be reproven to obtain the required number of quantifiers.
In particular, the sets $T([a^2, b)_K)$ are $\exists_6\La_K$-definable, and the sets $J^d([a^2, b)_K)$ and $H^a([a^2, b)_K)$ are $\exists_{13}\La_K$-definable and $\exists_{12}\La_K$-definable respectively.
A finite intersection of valuation rings $\bigcap_{v \in S} \mathcal{O}_v$ is $\exists_6\La_K$-definable, whereby $\Phi_u^S$ is $\exists_{11}\La_K$-definable.
The reader can verify these claims and that these eventually lead to the claimed number of quantifiers by counting the number of times conjunctions of existential formulas are taken in the construction of all these formulas.
\end{proof}
\section{Finitely generated subrings}\label{sectfingen}
\Cref{nicomainthmgev} shows us that rings of $S$-integers (with $S\subseteq \mathcal{V}_K$ finite) have a universal definition in their fraction field. What follows is a sketch on how to obtain from this a universal definition for any finitely generated ring in its global fraction field $K$.

For the rest of this section, let $R$ be a domain with a global field $K$ as its fraction field. Let $R'$ be the integral closure of $R$ in $K$.
\begin{lem}\label{fingenS}
Let $S = \lbrace v \in \mathcal{V}_K \mid R \not\subseteq \mathcal{O}_v \rbrace$. The following hold:
\begin{enumerate}
\item For any non-zero ideal $I$ of $R$, $R/I$ is a finite ring.
\item $R' = \mathcal{O}_S$.
\end{enumerate}
Furthermore, if $R$ is finitely generated, then the set $S$ is finite.
\end{lem}
\begin{proof}
Let $p = \charac(K)$. Set $R_0 = \mathbb{Z}$ if $p = 0$; if $p > 0$, let $R_0$ be a fixed subring of $R$ isomorphic to $\mathbb{F}_p[T]$. Let $K_0 = \Frac(R_0)$. Note that $R_0$ is a principal ideal domain and $K/K_0$ is a finite extension; by the Krull-Akizuki Theorem \cite[Proposition VII.2.5.5]{Bou06} $R$ is a noetherian ring of Krull dimension $1$ and $R'$ is a Dedekind domain. Furthermore, for any non-zero ideal $I$ of $R$, $I_0 = I \cap R_0$ is non-zero and $R/I$ is a finitely generated $R_0/I_0$-module. But $R_0/I_0$ is a finite ring, whereby also $R/I$ is finite. This concludes the proof of the first part.

Clearly by definition $R \subseteq \mathcal{O}_S$ and hence also $R' \subseteq \mathcal{O}_S$, as the latter is integrally closed. Since $R'$ is a Dedekind domain, it is the intersection of the discrete valuation rings in which it is contained \cite[Proposition 2.1]{Cas67-1}; hence $R' = \mathcal{O}_S$ by construction.

Finally, assume that $R$ is generated as a ring by $b_1, \ldots, b_n \in K$. Then $S$ consists precisely of those $v \in \mathcal{V}_K$ for which $v(b_i) < 0$ for at least one $i \in \lbrace 1, \ldots, n \rbrace$. There are only finitely many such valuations.
\end{proof}
\begin{opm}
One can show that an integral domain with global fraction field $K$ is finitely generated if and only if it is contained in $\mathcal{O}_S$ for some finite set $S \subseteq \mathcal{V}_K$. 
\end{opm}
\begin{lem}\label{fingenlem}
Assume that $R$ is finitely generated. There exists $r \in R \setminus \lbrace 0 \rbrace$ such that $rR' \subseteq R$.
\end{lem}
\begin{proof}
%By \Cref{Sintegerscharac} also $R'$ is finitely generated as a ring. As $R'/R$ is an integral extension, it follows that $R'$ is finitely generated as an $R$-module;
As $R$ is finitely generated as a $\zz$-algebra, its integral closure $R'$ is finitely generated as an $R$-module \cite[Corollary 7.7.4]{EGAIV}. Let $R' = Ra_1 + \ldots + Ra_n$ for $a_1, \ldots, a_n \in R'$. Since $R$ and $R'$ have the same fraction field, we have that for all $i$, there exists an $r_i \in R \setminus \lbrace 0 \rbrace$ such that $r_ia_i \in R$; setting $r = r_1\cdots r_n$ now yields $rR' \subseteq R$.
\end{proof}
\begin{stel}\label{fingenmainthm}
Assume that $R$ is finitely generated. Then $R$ has a $\forall\La_K$-definition in~$K$.
\end{stel}
\begin{proof}
By \Cref{fingenlem} there exists an $r \in R \setminus \lbrace 0 \rbrace$ such that $rR' \subseteq R$. As $rR'$ is a non-zero ideal of $R$, $R/rR'$ is finite by the first part of \Cref{fingenS}. Thus, there exist $y_1, \ldots, y_n \in R$ such that $R = \bigcup_{i=1}^n (y_i + rR')$. We have by the second part of \Cref{fingenS} that $R'$ is a ring of $S$-integers; it follows from \Cref{nicomainthmgev} that it has a $\forall\La_K$-definition in $K$. Since a finite union of $\forall\La_K$-definable sets is again $\forall\La_K$-definable; we conclude that $R$ is $\forall\La_K$-definable in $K$.
\end{proof}
\begin{opm}
This method does not give us any uniform bound on the number of quantifiers needed to define a finitely generated subring of a global field in its fraction field. To see that this is the case, consider for a non-zero $n \in \mathbb{N}$ the subring $\mathbb{Z}[ni]$ of $\mathbb{Q}[i]$ where $i^2 = -1$. Then the integral closure of $\mathbb{Z}[ni]$ in $\mathbb{Q}[i]$ is $\mathbb{Z}[i]$. One verifies that
\begin{displaymath}
\forall r \in \mathbb{Z}[i] : (r\mathbb{Z}[i] \subseteq \mathbb{Z}[ni] \Rightarrow \lvert \mathbb{Z}[ni]/r\mathbb{Z}[i] \rvert \geq n).
\end{displaymath}
Indeed, whenever $r \in \mathbb{Z}[i]$ satisfies $a\mathbb{Z}[i] \subseteq \mathbb{Z}[ni]$, then we must have $r \in n\mathbb{Z}[i]$. And $\lvert \mathbb{Z}[ni]/n\mathbb{Z}[i] \rvert = n$. If $\mathbb{Z}[i]$ has a $\forall_m\La_{\qq[i]}$-definition in $\mathbb{Q}[i]$, the technique from \Cref{fingenmainthm} gives us a $\forall_{nm}\La_{\qq[i]}$-definition for $\mathbb{Z}[ni]$ in $\mathbb{Q}[i]$.
\end{opm}
\begin{ques}
Can we give a uniform bound on the number of quantifiers needed to universally define a finitely generated subring of a global field in its fraction field?
\end{ques}
\begin{ques}
Let $R$ be a finitely generated domain. Does $R$ have a universal definition in its fraction field?
\end{ques}

%%%\bibliography{\LOCAL/masterref.bib}

\Addresses

\end{document}